\documentclass[12pt,a4paper]{amsproc}
\usepackage{latexsym}
\usepackage{amssymb}
\usepackage{amscd}
\usepackage[cp850]{inputenc}
\usepackage{mathrsfs}
\theoremstyle{plain}

\newtheorem{theorem}{Theorem}
\newtheorem{proposition}{Proposition}
\newtheorem{corollary}{Corollary}
\newtheorem{lemma}{Lemma}
\theoremstyle{remark}

\newcommand{\R}{\mathbb{R}}
\newcommand{\N}{\mathbb{N}}

\newcommand{\C}{\mathbb{C}}
\newcommand{\e}{\varepsilon}
\newcommand{\Z}{\mathscr{X}}
\newcommand{\X}{\mathscr{X}}

\def\Ext{\operatorname{Ext}}

\def\Ext{\operatorname{Ext}}

\begin{document}

%\noindent{\footnotesize \boxed{\text{The very pacific version: September 29, 2013}}}

\bigskip

\bigskip

\title{Complex interpolation and twisted twisted Hilbert spaces}
\author{F\'elix Cabello S\'anchez}
\author{Jes\'{u}s M. F. Castillo}
\address{Departamento de Matem\'aticas\\ Universidad de Extremadura\\
Avenida de Elvas\\ 06071-Badajoz\\ Spain} \email{fcabello@unex.es, castillo@unex.es}
\author{Nigel J. Kalton $\dag$}
%\address{Department of Mathematics\\ University of Missouri\\ Columbia MO6521, USA} \email{nigel@math.missouri.edu}

\thanks{This research has been supported in part by project
MTM2010-20190-C02-01 and the program Junta de Extremadura GR10113
IV Plan Regional I+D+i, Ayudas a Grupos de Investigaci\'on,}

%\subjclass{Primary: }

\maketitle

\begin{abstract}  {\bf We show that Rochberg's generalizared interpolation spaces $\Z^{(n)}$
arising from analytic families of Banach spaces form exact
sequences $0\to  \Z^{(n)} \to \Z^{(n+k)} \to \Z^{(k)} \to 0$. We
study some structural properties of those sequences; in particular, we show
that nontriviality, having strictly singular quotient map, or
having strictly cosingular embedding depend only on the basic case
$n=k=1$. If we focus on the case of Hilbert spaces obtained from
the interpolation scale of $\ell_p$ spaces, then $\Z^{(2)}$
becomes the well-known Kalton-Peck $Z_2$ space; we then show that
$\Z^{(n)}$ is (or embeds in, or is a quotient of) a twisted Hilbert
space only if $n=1,2$, which solves a problem posed by David
Yost; and that it does not contain $\ell_2$ complemented unless
$n=1$. We construct another nontrivial twisted sum of $Z_2$ with itself that
contains $\ell_2$ complemented.}\end{abstract}

\markboth{\underline{F Cabello S\'anchez, JMF Castillo, NJ Kalton}}{\underline{Complex Interpolation and Twisted Twisted Hilbert Spaces}}

\section{Introduction}

In 1979, Kalton and Peck developed a method to produce
nontrivial self-extensions of most quasi-Banach spaces with
unconditional basis \cite[Section 4]{kaltpeck}, including all Banach spaces apart from $c_0$; see \cite[Theorem~1]{ccs}.
 The most glaring examples are
perhaps the so-called $Z_p$ spaces, which are twisted sums of the
$\ell_p$ spaces. If, however, one wants to construct twisted sums
of $Z_p$, the Kalton-Peck's method simply does not work because of
their poor unconditional structure. On the other hand, the
existence of such twisted sums is guaranteed by the local theory
of exact sequences, at least when $p>1$; see e.g., \cite{cabecastuni}. Our starting goal with this
paper was to develop a method to obtain twisted sums of twisted
sum spaces, keeping the $Z_p$ spaces as the control
case.

%Interpolation theory enters here: as it is well-known
%\cite{kal-mon}, the Kalton-Peck spaces $Z_p$ can be obtained as
%the so-called derived spaces generated from the interpolation
%scale of $\ell_p$ spaces. This opens the door to the possibility of obtaining nontrivial twisting of Banach spaces generated by an interpolation scale, even when no unconditional structure is present. Of course, something else has to be done since the nontriviality of the twisting has to be worked apart.

%On the other hand it stands the representation problem for such new twisted spaces.

The path connecting interpolation theory and twisted sums was opened by Rochberg and Weiss, who introduce in
\cite{rochberg-weiss} certain spaces which naturally arise in the
study of ``analytic families'' of Banach spaces and that turn out to be twisted sums of the ``intermediate'' spaces.
Actually, if $\mathscr F$ is
the usual Calder\'on space of analytic functions on the strip
$0<\Re z<1$ associated to the couple $(\ell_\infty,\ell_1)$ in the
complex interpolation method then, as it is well-known, $
[\ell_\infty,\ell_1]_\theta=\{f(\theta): f\in\mathscr F\}=\ell_p,
$ where $p=1/\theta$ and $0<\theta<1$ and
\begin{equation}\label{zp}
Z_p=\{(f'(\theta), f(\theta)): f\in\mathscr F\},
\end{equation}
with the quotient norm inherited from $\mathscr F$ ---though this
is not made explicit in \cite{rochberg-weiss}.

Nothing seems
to prevent one from adding more derivatives to (\ref{zp}) and
figure out that the resulting space represents the iterated twisted sum
spaces. Such is exactly what Rochberg did in \cite{rochberg}, in
the broader setting of analytic families of Banach spaces. Performing that is not, by far, as simple as it sounds; and
perhaps the turning point in Rochberg's approach is the using of
Taylor coefficients instead of merely putting
derivatives, as it is suggested in \cite[Section 10, p.
1161]{kal-mon}.

Such approach is the one we adopt in this paper, which in this
regard can well be considered a spin-off from Rochberg's
\cite{rochberg}; with several variations, the first of which is
the use of admissible spaces of analytic functions instead of
analytic families, what makes ``reiteration'' both unavailable and
unnecessary. Thus, given an admissible space of analytic functions
$\mathscr F$, we consider the space $\Z^{(n)}$ of all possible
lists of Taylor coefficients of functions in $\mathscr F$ of
lenght $n$ -- at a fixed point $z$ which is understood from now on
-- endowed with the obvious infimum norm on it. Then we observe
that if $m=n+k$, there is an exact sequence
\begin{equation}\label{nmk}
\begin{CD}
0@>>> \Z^{(n)} @>>> \Z^{(m)}@>>> \Z^{(k)} @>>>0
\end{CD}
\end{equation}
and so $\Z^{(m)}$ is a twisted sum of $\Z^{(n)}$ and $\Z^{(k)}$. The
key nontrivial step here is obtaining the right form of the
embedding. To this we devote section \ref{maps} in which we
obtain two (equivalent) representations for the embedding, depending on the
representation of the spaces. Regarding the sequences themselves, we will
show that many properties, such as nontriviality,
having strictly singular quotient map, or having strictly
cosingular embedding depend only on the seed case $n=k=1$.
The nontriviality of this case has to be worked apart.

We then focus on the case in which $\mathscr F$ is the Calder\'on
space of the couple $(\ell_\infty,\ell_1)$. If we fix $z={1\over
2}$, then $\Z^{(1)}=\ell_2$ and $\Z^{(2)}$ is the Kalton-Peck $Z_2$ space
\cite{kaltpeck}. The space $\Z^{(3)}$ is both a twisted sum
of $\ell_2$ with $Z_2$ and a twisted sum of $Z_2$ with $\ell_2$,
and $\Z^{(4)}$ is, among other possibilities, a twisted sum of
$Z_2$ with itself, as desired. We then pass to establish
structural properties of the spaces $\Z^{(n)}$ and of the sequences
(\ref{nmk}). Regarding the spaces, we will show that $\Z^{(n)}$ is
(or embeds in, or is a quotient of) a twisted Hilbert space only
if $n=1,2$ --which solves a problem posed by David Yost-- and that
it does not contain $\ell_2$ complemented unless $n=1$. To put
this result in perspective, we will construct a nontrivial twisted
sum of $Z_2$ with itself that contains $\ell_2$ complemented.

\section{Preliminaires}

We warmly recommend the reader who is not familiar with
Kalton and Peck paper \cite{kaltpeck} or Rochberg's
\cite{rochberg} to postpone this article until get acquainted with
them. Perusing the papers \cite{cjmr, r-uses}, the article
\cite{kal-mon} in the Handbook and the monograph \cite{castgonz}
can help with the background. Anyway, the basic ingredients to
read this paper are operatively described next.
\subsection{Exact sequences} A short sequence of Banach spaces and (linear, bounded) operators
\begin{equation}\label{abc}
\begin{CD}
0@>>> A @>I>> B@>Q>> C @>>>0\end{CD}
\end{equation}
is said to be exact if the kernel of each arrow equals the image of the preceding one.
As $I(A)$ is closed in $B$ the operator $I$ embeds $A$ as a subspace of $B$ and
$C$ is isomorphic to the quotient $B/I(A)$, by the open mapping theorem.
For this reason one often says that $B$ is a twisted sum of $A$ with $C$ (in that order); the whole sequence (\ref{abc}) is said to be an extension of $C$ by $A$ (the order was reversed by ``functorial'' reasons).

The extension (\ref{abc}) is said to be trivial if there is an
operator $P:B\to A$ such that $P\circ I={\bf 1}_A$ (i.e., $I(A)$
is complemented in $B$); equivalently, there is an  operator
$J:C\to B$ such that $Q\circ J={\bf 1}_C$. In this case $P\times
Q:B\to A\times C$ is an isomorphism, with inverse $I\oplus J$ and
thus the ``twisted sum'' $B$ is (isomorphic to) the direct sum
$A\oplus C=A\times C$.

\subsection{Admissible spaces of analytic functions}
We will work within the framework of an admissible space of
analytic functions as defined by Kalton and Montgomery-Smith in
\cite[Section~10]{kal-mon}. So, let $U$ be an open set of $\C$
conformally equivalent to the disc $\mathbb D=\{z\in\mathbb C: |z|<1\}$ and $W$ a complex
Banach space. A Banach space $\mathscr F$ of analytic functions
$F:U\to W$ is said to be admissible provided:
\begin{itemize}
\item[(a)] For each $z\in U$, the evaluation map
$\delta_z:\mathscr F\to W$ is bounded. \item[(b)] If
$\varphi:U\to\mathbb D$ is a conformal equivalence, then
$F\in\mathscr F$ if and only if $\varphi\cdot F\in\mathscr F$ and
$\|\varphi\cdot F\|_{\mathscr F}= \|F\|_{\mathscr F}$.
\end{itemize}

For each $z\in U$ we define $X_z=\{x\in W: x=F(z) \text{ for some
} F\in\mathscr F\}$ with the norm $\|x\|=\inf\{\|F\|_{\mathscr F}:
x=F(z)\}$ so that  $X_z$ is isometric to $\mathscr F/\ker\delta_z$
with the quotient norm. One often says that $(X_z)_{z\in U}$ is an
analytic family of Banach spaces. The simplest examples arise from
complex interpolation. Indeed, let $(X_0,X_1)$ be a Banach couple
and take $W=X_0+X_1$ and $U$ the strip $0<\Re z<1$. Let $\mathscr
F=\mathscr C(X_0,X_1)$ be the Calder\'on space of those continuous
functions $F:\overline U\to W$ which are analytic on $U$ and
satisfy the boundary conditions that, for $k=0,1$ one has
$F(k+ti)\in X_k$ and $\|F\|_{\mathscr C}=\sup\{\|F(k+ti)\|_{X_k}:
t\in\R, k=0,1\}<\infty$. Then $\mathscr F$ is admissible and
$X_z=[X_0,X_1]_\theta$, with $\theta=\Re z$, is an analytic
family.

It is important now to realize that when $\mathscr F$ is
admissible, then the map $\delta_z^{n}:\mathscr F\to W$,
evaluation of the $n$-{th} derivative at $z$, is bounded for all
$z\in U$ and all $n\in \N$ by an iterated use of (a), the
definition of derivative and the principle of uniform boundedness.
Thus, it makes sense to consider the Banach spaces
\begin{equation}\label{Fker}
\mathscr F/\bigcap_{i<n}\ker\delta_{z}^{i}\quad\quad(n\in\N).
\end{equation}
%Of course, assuming that $\delta_z^0 = \delta_z$ then $\mathscr F/\ker \delta_{z}^{0} =X_z$.

\section{Exact sequences of derived spaces}\label{maps}

\subsection{Lists of Taylor coefficients}\label{taylor} Following Rochberg, let
us fix $z\in U$ and consider the following spaces:
$$
\Z^{(n)}_{z}= \{(x_{n-1},\dots,x_{0})\in W^n: x_i=\widehat
f[i;z]\text{ for some } f\in \mathscr F \text{ and all } 0\leq i<n
\},
$$
where $\widehat f[i;z]=f^{(i)}(z)/i!$ is the $i$-th Taylor
coefficient of $f$ at $z$. Thus, the elements of $ \Z^{(n)}_{z}$
are (truncated) sequences of Taylor's coefficients (at $z$) of
functions in $\mathscr F$ arranged in decreasing order. Here, we
deviate from Rochberg notation in two points: first, the
superscript $(n)$ refers to the ``number of variables'' and not to
the highest derivative, and second, we have arranged Taylor
coefficients decreasingly in order to match with the usual
notation for twisted sums, with the subspace on the left and the
quotient on the right. If we equip $\Z^{(n)}_{z}$ with the obvious
quotient norm, it is isometric to $\mathscr
F/\bigcap_{i<n}\ker\delta_{z}^{i}$ via Taylor coefficients and so it is complete.
From now on we shall omit the
base point $z$, which is understood to be fixed.

\subsection{Operators}
We introduce next certain ``natural'' operators linking the
various spaces $\Z^{(n)}$ as $n$ varies. Those operators will be used to construct the exact sequences we want.

To this end, for $1\leq n,
k<m$ we denote by $\imath_{n,m}:W^n\to W^m$ the inclusion on the
left given by
$\imath_{n,m}(x_n,\dots,x_1)=(x_n,\dots,x_1,0\dots,0)$ and by
$\pi_{m,k}:W^m\to W^k$ the projection on the right given by
$\pi_{m,k}(x_m,\dots, x_k,\dots,x_1)=  (x_k,\dots,x_1)$. While
$\pi_{m,k}$ is obviously a quotient map from $\Z^{(m)}$ onto
$\Z^{(k)}$, it is not clear at all that $\imath_{n,m}$ maps
$\Z^{(k)}$ to $\Z^{(n)}$, let alone its continuity. To prove that
this is indeed the case we need some extra work.

Observe that if $\varphi$ is as in (b) and if $\phi$ is a
``polynomial'' in $\varphi$, that is, $\phi=\sum_ia_i\varphi^i$
for some finite sequence of complex numbers $(a_i)$, then $\phi\cdot f\in \mathscr F$ for each $f\in\mathscr F$ and
$\|\phi\cdot f\|_{\mathscr F}\leq
(\sum_i|a_i|)\|f\|_{\mathscr F}$.

\begin{lemma}\label{phi}
Let $\varphi:U\to\mathbb D$ be a conformal equivalence vanishing
at $z$. Then, for  $0\leq k\leq m$ there is a polynomial $P$ of
degree at most $m$ such that $\widehat {P\circ \varphi}[
i;z]=\delta_{ik}$ for every $0\leq i\leq m$.
\end{lemma}

\begin{proof} If $f:U\to\mathbb C$ is holomorphic, then $f\circ\varphi^{-1}$ is
holomorphic on the disk and we have
$$
f(\varphi^{-1}(w))=\sum_{n=0}^\infty a_n w^n\quad\quad(|w|<1),
$$
where $a_n$ is the $n$-th Taylor coefficient of
$f\circ\varphi^{-1}$ at the origin. In particular
$f\circ\varphi^{-1}$ has a contact of order $m$ with the
polynomial defined by $P(w)=\sum_{n=0}^ma_nw^n$ at the origin. As
$\varphi$ is a conformal equivalence we have that
$f=f\circ\varphi^{-1}\circ\varphi$ has a contact of order $m$ with
the function
$$
P\circ\varphi= \sum_{n=0}^ma_n\varphi^n
$$
at $z=\varphi^{-1}(0)$. In particular the first $m$ derivatives of
$f$ and $\sum_{n=0}^ma_n\varphi^n$ agree at $z$. The Lemma follows
just applying this construction to the function $f(w)=(w-z)^k$.
\end{proof}

The following Proposition is a slight generalization of \cite[Proposition 3.1]{rochberg}:

\begin{proposition}\label{operators} Suppose $1\leq n, k<m$. Then:
\begin{itemize}
\item[(a)] The map $\imath_{n,m}: \Z^{(n)} \to \Z^{(m)}$ is bounded.

\item[(b)] The map $\pi_{m,k}: \Z^{(m)} \to \Z^{(k)}$ is an ``isometric'' quotient.
\end{itemize}
\end{proposition}

\begin{proof} Part (b) is obvious. To prove (a) we must prove that there is a constant $M$ such that if
$(x_n,\dots,x_1)$ is the list of Taylor coefficients of some $f\in
\mathscr F$, then there is another $g\in\mathscr F$ whose
coefficients are $(x_n,\dots,x_1,0,\dots,0)$ with $\|g\|_{\mathscr
F}\leq M\|f\|_{\mathscr F}$. Set $k=m-n$ and apply Lemma~\ref{phi}
to get a polynomial $\phi=\sum_{i=0}^{m-1}a_i\varphi^i$ such that
$\widehat\phi[i;z]=\delta_{ik}$, for $0\leq i<n+k$ and take
$M=\sum_{i=0}^{n+k-1}|a_i|$. Now, if $f\in\mathscr F$ and $g=\phi
f$, then $\|g\|_{\mathscr F} \leq M\|f\|_{\mathscr F}$. Moreover,
for $i\in[0,n+k)$, one has
$$
\widehat g[i;z]=\widehat{(\phi f)}[i;z]=\sum_{j=0}^i \widehat\phi[j;z]\cdot \widehat f[i-j;z],
$$
by Leibniz rule. Hence $\hat g[i;z]=0$ if $i<k$ and for $i\geq k$
we have $ \widehat g[i;z]=\widehat f[i-k; z], $
 as required.
\end{proof}

\subsection{Exactness}
From now on we will omit the names $\imath_{m,n}$ and $\pi_{m,k}$
and so unlabelled arrows $\Z^{(n)}\to \Z^{(m)}$ must be understood
to be $\imath_{n,m}$ if $n\leq m$ and $\pi_{n,m}$ for $n\geq m$,
unless otherwise declared. With these conventions the aim of this Section is to prove that, given integers $n$ and $k$, the ``obvious'' sequence
 $0\longrightarrow \Z^{(n)}\longrightarrow
\Z^{(n+k)}\longrightarrow \Z^{(k)}\longrightarrow 0$ is exact.

First of all,  observe that
the various possible sequences passing through a given $\Z^{(m)}$
are compatible in the sense that if $m=k+n=i+j$, with $k<i$ say,
then the following diagram is commutative
\begin{equation}\label{commutative}
\begin{CD}
\Z^{(j)}@= \Z^{(j)}\\
@VVV @VVV \\
 \Z^{(n)}@>>>  \Z^{(m)}@>>>  \Z^{(k)}\\
@VVV@VVV@|\\
 \Z^{(n-j)}@>>>  \Z^{(i)}@>>>  \Z^{(k)}
\end{CD}
\end{equation}

%\textbf{The second author believes that the following lemma can be
%considered the abstract version of the ``compatible interpolators"
%of CCS \cite{caceso}; the first one, doesn't. But he also
%disbelieves that dragon's blood cures love diseases.}

The key point is isolated in the next lemma.% is the key to show that the lower sequence is exact, which is perhaps the main result of this Section.

\begin{lemma}\label{content}
If $(x, 0, \dots, 0)\in \Z^{(k+1)}$, then $x\in \Z^{(1)}$.
\end{lemma}
\begin{proof} Pick $x\in W$ and suppose $(x,0,\dots,0)\in \Z^{(k+1)}$. % We must check that $x\in \Z^{(1)}$.
Let us take $f\in \mathscr F$ such that $\widehat f[i;z]=0$ for $i<k$
and $x=\widehat f[k;z]$. Then $f$ has a zero of order $k-1$ at $z$ and
it can be written as $f=\varphi^{k} g$, where $g:U\to W$ is
analytic. It follows from (b) that $g\in\mathscr F$ and
$\|g\|_\mathscr F=\|f\|_\mathscr F$. But
$$x=\hat f[k;z]=\widehat{(\varphi^{k} g)}[k;z]=\sum_{i=0}^k\widehat{(\varphi^k)}[i;z]\cdot\widehat g[k-i;z]=\frac{(\varphi^k)^{(k)}(z) g(z)}{k!}=
\varphi'(z)^kg(z)$$ and so $x\in \Z^{(1)}$. \end{proof}

%Consequently:

\begin{theorem}\label{exact}
The sequence $0\longrightarrow \Z^{(n)}\longrightarrow
\Z^{(n+k)}\longrightarrow \Z^{(k)}\longrightarrow 0$ is exact.
\end{theorem}

\begin{proof} The proof proceeds by induction on $m=n+k$. The
previous lemma shows that for every $m\in \N$, the sequence
$0\longrightarrow \Z^{(1)} \longrightarrow Z^{(m)} \longrightarrow
\Z^{(m-1)}\longrightarrow 0$ is exact. By the induction hypothesis,
the sequence $0\longrightarrow \Z^{(n-1)} \longrightarrow \Z^{(m-1)}
\longrightarrow \Z^{(k)}\longrightarrow 0$ is also exact. The
compatibility of such sequences yields the commutative diagram
$$
\begin{CD}
&&0&&0\\ &&@VVV @VVV\\
&& \Z^{(1)}@= \Z^{(1)}\\
&&@VVV @VVV \\
&&\Z^{(n)}@>>>  \Z^{(m)}@>>>  \Z^{(k)}\\
&&@VVV@VVV@|\\
0@>>> \Z^{(n-1)}@>>>  \Z^{(m-1)}@>>>
\Z^{(k)}@>>>0\\
&&@VVV@VVV\\
&&0&&0
\end{CD}
$$and a simple chasing of arrows shows that the middle
sequence must also be exact.\end{proof}

\begin{corollary} If $(x_n,\dots,x_1, 0,\dots, 0)\in \Z^{(n+k)}$, then
$(x_n,\dots,x_1)\in \Z^{(n)}$.
\end{corollary}

This implies that $\|(x_n,\dots,x_1, 0,\dots, 0)\|_{\Z^{(n+k)}}$ is
equivalent to $\|(x_n,\dots,x_1)\|_{\Z^{(n)}}$, although we will
not pursue
any bound here.\\

A new look can be paid now at Diagram (\ref{commutative}) to
exploit its form to study the splitting of the exact sequences it
contains. After Theorem \ref{exact} the diagram has become
\begin{equation}\label{poz}
\begin{CD}
&&0&&0\\ &&@VVV @VVV\\
&&\Z^{(j)}@= \Z^{(j)}\\
&&@VVV @VVV \\
 0@>>> \Z^{(n)}@>>>  \Z^{(m)}@>>>  \Z^{(k)}@>>>0\\
&&@VVV@VVV@|\\
 0@>>>\Z^{(n-j)}@>>>  \Z^{(i)}@>>>  \Z^{(k)}@>>> 0\\
&&@VVV@VVV\\
&&0&&0.
\end{CD}
\end{equation}
Thus, if the middle horizontal sequence splits, then so does the
lower one, and if the middle vertical sequence splits, then so
does the vertical sequence on the left. Putting together these two
pieces one gets the commutative diagram
$$
\begin{CD}
0@>>> \Z^{(n)}@>>>  \Z^{(n+k)}@>>>  \Z^{(k)} @>>>0 \\
&&@VVV @VVV @|\\
0@>>>\Z^{(1)}@>>>  \Z^{(k+1)}@>>>  \Z^{(k)}@>>> 0\\
&&@| @AAA@AAA\\
0@>>> \Z^{(1)}@>>> \Z^{(2)}@>>>  \Z^{(1)}@>>> 0\\
\end{CD}
$$
%\vspace{0,3cm}
from which it immediately follows

\begin{corollary}\label{nonsplitting}
If the sequence $ 0\longrightarrow \Z^{(n)}\longrightarrow
\Z^{(n+k)}\longrightarrow  \Z^{(k)}\longrightarrow 0 $ is nontrivial
for $n=k=1$, then it is nontrivial for all integers $n$ and $k$.
\end{corollary}

\subsection{An isometric variant}
There is another form for the exact sequences $0\to \X^{(n)}\to \X^{(n+k)}\to \X^{(k)}\to 0$ which is even easier to describe in abstract terms. Consider again the quotient spaces
$$
Q^{(n)}_z= \mathscr F/\bigcap_{i<n}\ker\delta_{z}^{i}\quad\quad(n\in\N).
$$
These spaces are isometric to the corresponding $\X^{(n)}_z$ via Taylor coefficients, but we do not need this fact at this moment. Let us fix integers $n$ and $k$.
It is clear that there is a natural quotient map from $Q^{(n+k)}_z$ onto $Q^{(k)}_z$ that we shall not even label.
Less obvious is that the kernel of this map is isometric to  $\X^{(n)}_z$, although this time the isometry is not ``natural''.
To see this let us fix a conformal equivalence $\varphi:U\to\mathbb D$ having a (single) zero at $z$. (We observe that if $\phi$ is another conformal equivalence with $\phi(z)=0$, then $\phi=\lambda\varphi$, where $\lambda\in\mathbb T$; thus $\varphi$ is unique if we insist that $\varphi'(0)$ is real and positive.) Now recall that $f\in \bigcap_{i< k}\ker\delta_{z}^{i}$ if and only if there is a (necessarily unique) $g\in \mathscr F$ such that $f=\varphi^k g$ and one has $\|f\|_{\mathscr F}=\|g\|_{\mathscr F}$, by (b). It is therefore clear that the map $f\in\mathscr F\mapsto \varphi^{k}f\in\mathscr F$ induces an isometry of  $Q^{(n)}_z$ into  $Q^{(n+k)}_z$ whose range is $\ker(Q^{(n+k)}_z\to Q^{(k)}_z)$.

Thus the space
$Q^{(n+k)}_z$ is an ``isometric'' twisted sum of $Q^{(n)}_z$ and $Q^{(k)}_z$. More precisely, the short sequence
\begin{equation}\label{isometricQ}
\begin{CD}
0@>>> Q^{(n)}_z@>{\varphi^k\cdot}>> Q^{(n+k)}_z@>>> Q^{(k)}_z@>>> 0
\end{CD}
\end{equation}
is exact. We will omit from now on the base point $z$, which is understood. As before, the decompositions of a given $Q^{(m)}_z$ into twisted sum of the preceding  spaces $Q^{(n)}_z$
are all compatible in the sense that if $m=k+n=i+j$, with $k<i$, then the following diagram is commutative
$$
\begin{CD}
Q^{(j)}@= Q^{(j)}\\
@V{\varphi^{j-n}\cdot}VV @VV{\varphi^i\cdot}V \\
Q^{(n)}@>{\varphi^k\cdot}>>Q^{(m)}@>>> Q^{(k)}\\
@VVV@VVV@|\\
Q^{(j-n)}@>{\varphi^{k}\cdot}>> Q^{(i)}@>>> Q^{(k)}
\end{CD}
$$

It is interesting to compare the sequence (\ref{isometricQ}) to that appearing in Theorem~\ref{exact}. To this end, we observe that, after identifying $\mathscr X^{(m)}$ and $Q^{(m)}$ through Taylor coefficients, the operator $\mathscr X^{(n)}\to \mathscr X^{(n+k)}$ which corresponds to
$\imath_{n,n+k}: \X^{(n)}\to \X^{(n+k)}$ is just multiplication by $\phi$, where $\phi$ is the polynomial appearing in the proof of Proposition~\ref{operators}(a), that is, $\phi=\sum_{0\leq i<n+k}a_i\varphi^i$, with $\widehat \phi[i;z]=\delta_{ik}$ for $0\leq i<n+k$. Clearly, $a_i=0$ for $0\leq i<k$ and so $\phi=\varphi^k\psi$, where $\psi=\sum_{k\leq i<n+k}a_i\varphi^{i-k}$. Thus the following diagram is commmutative
$$
\begin{CD}
Q^{(n)}@>{\varphi^k\cdot}>> Q^{(n+k)}@>>> Q^{(k)}\\
@|@V\psi\cdot VV@|\\
Q^{(n)}@>{\phi\cdot}>> Q^{(n+k)}@>>> Q^{(k)}\\
@V{\widehat\cdot }VV @V{\widehat\cdot }VV @V{\widehat\cdot }VV\\
\X^{(n)}@>{\imath_{n,n+k}}>>  \X^{(n+k)}@>>>  \X^{(k)}
\end{CD}
$$
It follows from the 3-lemma (see for instance \cite[Lemma 1.1]{hiltstam}) and the open mapping theorem that multiplication by $\psi$ induces an automorphism of  $Q^{(n+k)}$ and so, in the preceding diagram, the first row is equivalent to the second one, and both are ``isomorphically equivalent'' (in the language of \cite[p. 256]{castmoreisr}) to the third one; which means that the three sequences have the same ``isomorphic" properties.

\subsection{The space $\Z^{(n+k)}$ as a twisted sum of $\Z^{(n)}$ and $\Z^{(k)}$}\label{twistedsum}

It is a part of the by now classical theory of twisted sums as
developed by Kalton (see \cite[Proposition 3.3]{kal-loc} or
\cite[Theorem 2.4]{kaltpeck}) that if $A$ and $C$ are Banach or
quasi-Banach spaces, then every short exact sequence $0\to A\to
B\to C\to 0$ arises, up to equivalence, from a quasilinear map
from $C$ to $A$. Thus, in view of Theorem~\ref{exact}, given integers $k$ and $n$, there
must be some quasilinear map $\Omega_{k,n}$ associated to the
exact sequence $0\to \Z^{(n)}\to \Z^{(n+k)}\to \Z^{(k)}\to 0$. From
an abstract point of view, the description of $\Omega_{k,n}$ is
rather easy. One fixes some (small) $\e>0$. Given
$x=(x_{k-1},\dots, x_{0})\in \Z^{(k)}$ we select (homogeneously)
$f\in \mathscr F$ such that $\|f\|\leq(1+\e)\|x\|_{\mathscr
Z_{(n)}}$ and $\widehat f[i;z]=x_{i}$ for $0\leq i<k$ and we
define $\Omega_{k,n}: \Z^{(k)}\to W^n$ by letting
\begin{equation}\label{asin}
\Omega_{k,n}(x)=(\widehat f[n+k-1;z],\dots,\widehat f[k;z]).
\end{equation}
Following the uses of the theory, the twisted sum space (sometimes
known as the derived space) is then defined by
$$
\Z^{(n)} \oplus_{\Omega_{k,n}}  \Z^{(k)}=\{(y,x)\in W^n\times W^k:
x\in  \Z^{(k)}, y-\Omega_{k,n}(x)\in \Z^{(n)} \},
$$
endowed with the quasinorm
\begin{equation}\label{quasinorm}
\|(y,x)\|_{{\Omega_{k,n}}}=\|y-\Omega_{k,x}(x)\|_{\Z^{(n)}}+\|x\|_{\Z^{(k)}}.
\end{equation}
Of course it has not yet been proved neither that $\Omega_{k,n}$
is quasilinear nor that the formula (\ref{quasinorm}) defines a
quasinorm. We may skip these steps since we have the following.

\begin{proposition}
The spaces $\Z^{(n)} \oplus_{\Omega_{k,n}} \Z^{(k)}$ and
$\Z^{(n+k)}$ are the same.
\end{proposition}

\begin{proof} Suppose $(y,x)=(y_{n-1},\dots,y_{0},x_{k-1},\dots,x_{0})\in \Z^{(n+k)}$ so that there is $F\in \mathscr F$ whose list of Taylor coefficients begins with $(y,x)$. % We may assume $\|F\|_{\mathscr F}\leq (1+\e)\|(y,x)\|_{\X^{(n+k)}}$.
Then $x\in \Z^{(k)}$ and $(\Omega_{k,n}(x),x)\in \Z^{(n+k)}$, so
$(y,x)-(\Omega_{k,n}(x),x)=(y- x,\Omega_{k,n}(x),0)$ belongs to
$\Z^{(n+k)}$ and by Lemma~\ref{content} we have
$y-\Omega_{k,n}(x)\in \Z^{(n)}$. Regarding  the involved norms, one
has
$$
\|y-\Omega_{k,x}(x)\|_{\Z^{(n)}}\leq
C(\|(y,x)\|_{\Z^{(n+k)}}-\|(\Omega_{k,n}(x),x)\|_{\Z^{(n+k)}})\leq
(C+1)\|(y,x)\|_{\Z^{(n+k)}},
$$
where $C$ is the constant implicit in Lemma~\ref{content}. Hence
$\|(y,x)\|_{\Omega_{k,n}}\leq (C+2)\|(y,x)\|_{\Z^{(n+k)}}$.

As for the other containment, suppose $(y,x)\in \Z^{(n)}
\oplus_{\Omega_{k,n}} \Z^{(k)}$, that is, $x\in \Z^{(k)}$ and
$y-\Omega_{k,n}(x)\in \Z^{(n)}$. Then if $f$ is the function
associated to $\Omega_{k,n}(x)$ as in (\ref{asin}) and
$g\in\mathscr F$ is almost optimal for $y-\Omega_{k,n}(x)\in
\Z^{(n)}$, taking $\phi$ as in Lemma~\ref{phi}, we have that
$(y,x)$ is the list of Taylor coefficients of $F=f+\phi\cdot g$,
so $(y,x)\in \Z^{(n+k)}$ and
$$
\|(y,x)\|_{\Z^{(n+k)}}\leq \| f+\phi\cdot g \|_\mathscr F\leq
(1+\e)\left( M\|y-\Omega_{k,n}(x)\|_{\Z^{(n)}}+
\|x\|_{\Z^{(k)}}\right),
$$
where $M$ is as in the proof of Proposition~\ref{operators}(a).
\end{proof}

\section{Singularity of the exact sequences of derived spaces}

Recall that an operator is said to be strictly singular if its
restriction to an infinite dimensional subspace of its domain is
never an isomorphism; and that an operator  $u: A\to B$ is
strictly cosingular if for every infinite codimensional subspace
$C$ of $B$ the composition $\pi\circ u:A\to B\to B/C$ fails to be
onto; equivalently, $u^*:B^*\to A^*$ is not an isomorphism when
restricted to any weakly* closed infinite-dimensional subspace of
$B^*$. Strictly singular operators were introduced by Kato
\cite{kato} and strictly cosingular by Pe\l czy\'nski \cite{pelc}.

An exact sequence is said to be singular when the quotient map is
strictly singular and will be called cosingular when the embedding
is strictly cosingular. We refer the reader to \cite{castmoresing,
ccs} for some steps into the theory of singular and cosingular
sequences. The Kalton-Peck sequences $0\to \ell_p\to Z_p\to
\ell_p\to 0$ are singular for all $p\in(0,\infty)$ and cosingular at least for $p\in(1,\infty)$. We
need the following result.
\begin{lemma}\label{3strict} Assume one has a commutative diagram
$$\begin{CD}
0@>>> A @>I>>  B @>Q>> C @>>> 0\\
 &&@VtVV @VVTV @| \\
0@>>> D@>>>  E@>>> C@>>>0
\end{CD}$$
with exact rows. If both $Q$ and $t$ are strictly singular then $T$ is strictly
singular.
\end{lemma}

\begin{proof}
We need the following characterization of strictly singular
quotient maps. Let $B$ be a Banach space and $A$ a closed subspace
of $B$. Then the quotient map $Q: B\to B/A$ is strictly singular
if and only if for every infinite-dimensional subspace $B'\subset
B$ there is an infinite-dimensional $A'\subset A$ and a compact
(actually nuclear) operator $K:A'\to B$ such that $I+K$ embeds
isomorphically $A'$ into $B'$. This maybe folklore; see
\cite[Proposition 3.2]{css} for an explicit proof. A certainly
classical result establishes that an operator $t:A\to D$ is
strictly singular if given any infinite dimensional subspace
$A'\subset A$ and $\varepsilon>0$ there is a further infinite
dimensional subspace $A''\subset A'$ such that
$\|t_{|A''}\|<\varepsilon$. Both things together yield that given
$B'\subset B$ there is $A''\subset A'\subset A$ such that $I+K:
A''\to B'$ is an into isomorphism and $\|t_{|A''}\|<\varepsilon$.
There is no loss of generality assuming that
$\|K_{|A''}\|<\varepsilon$. Therefore
$\|T_{|(I+K)(A'')}\| = \|t_{|A''} + TK_{|A''} \| <(1+ \|T\|)\varepsilon.$
\end{proof}

We thus obtain the ``strictly singular counterpart" to
Corollary~\ref{nonsplitting}:

\begin{proposition}\label{singular}
If the natural quotient map $\Z^{(2)}\to \Z^{(1)}$ is strictly
singular, then so is $\Z^{(n)}\to \Z^{(k)}$ for every $n>k$.
\end{proposition}
\begin{proof} Note that if $n>m>k$, then $ \Z^{(n)}\to \Z^{(k)}$ is $ \Z^{(n)}\to
\Z^{(m)}$ followed by $ \Z^{(m)}\to \Z^{(k)}$. As the composition of
a strictly singular operator with any operator is again strictly
singular, we have that the Proposition is trivial if $k=1$ and also
that one can assume $n=k+1$. We shall prove that $\Z^{(k+1)}\to
\Z^{(k)}$ is strictly singular by induction on $k\in\N$. There is
nothing to prove for $k=1$, so assume $k>1$. Since one has the
commutative diagram
$$
\begin{CD}
&&\Z^{(1)}@= \Z^{(1)}\\
&&@VVV @VVV \\
0@>>> \Z^{(k)}@>>> \Z^{(k+1)}@>\pi_{k+1,1}>> \Z^{(1)} @>>> 0\\
&&@V\pi_{k, k-1}VV@V \pi_{k+1,k}VV@|\\
0@>>> \Z^{(k-1)}@>>> \Z^{(k)}@>>> \Z^{(1)}@>>>0.
\end{CD}
$$
But $\pi_{k+1,1}$ is strictly singular and so is $\pi_{k,k-1}$, by the induction hypothesis. Thus, the result follows from Lemma~\ref{3strict}.
\end{proof}

We omit the proofs of the dual results:
\begin{lemma}\label{3costrict} Assume one has a commutative diagram
$$\begin{CD}
0@>>> A @>I>>  B @>Q>> C @>>> 0\\
 &&@| @AATA @AAtA  \\
0@>>> A@>>>  D@>>> E@>>>0
\end{CD}$$
with exact rows. If both $I$ and $t$ are strictly cosingular then $T$ is strictly
cosingular.
\end{lemma}

\begin{corollary}\label{cosingular}
If the inclusion map $\Z^{(1)}\to \Z^{(2)}$ is strictly cosingular,
then so is $\Z^{(k)}\to \Z^{(n)}$ for every $k<n$.
\end{corollary}

\section{Applications to Hilbert spaces}

\subsection{The quasi-linear map associated to twisted Kalton-Peck spaces }\label{KPR}

Some results in this section are, essentially, in
\cite[Section~6.B]{rochberg}. Let us consider the following
variation of the Calder\'on space associated to the Banach couple
$(\ell_\infty, \ell_1)$ which is designed to simplify the
computation of extremals. Take $U=\mathbb S=\{z\in \mathbb C:
0<\Re z< 1\}$, with $W=\ell_\infty$, and let $\mathscr F$ be the
space of analytic functions $F:\mathbb S\to\ell_\infty$ having the
following properties:
\begin{enumerate}
\item $F$ extends to a $\sigma(\ell_\infty,\ell_1)$ continuous function on $\overline{\mathbb S}$ that we denote again by $F$.
\item $\|F\|_\mathscr F=\sup\{\|F(it)\|_\infty,\|F(1+it)\|_1:t\in\mathbb R\}<\infty$.
\end{enumerate}
Let $(\mathscr Z_z)_{z\in\mathbb S}$ denote the analytic family induced by $\mathscr F$.
Then of course $\mathscr Z_z=[\ell_\infty,\ell_1]_{\theta}=\ell_p$, where $\theta=\Re z$ and $p=1/\theta$ for $\theta\in(0,1)$ and, in particular $\mathscr Z_z=\ell_2$ for $z=1/2$. In the remainder of this Section we fix $z=1/2$ as the base point.

If $x$ is normalized in $\ell_2$, then $F_x(z)=u|x|^{2z}$ is normalized in $\mathscr F$ (although it does not belong to $\mathscr C(\ell_\infty,\ell_1)$ in general) and one has $F_x({1\over 2})=x$, where $x=u|x|$ is the ``polar decomposition'' of $x$. Now
$$%\begin{align*}
F_x=u|x||x|^{2z-1}=x|x|^{2(z-1/2)}=x\sum_{n=0}^\infty\frac{2^n\log^n|x|}{n!}\left(z-\tfrac{1}{2}\right)^n,
$$%\end{align*}
and
$$
\widehat F_x[n, \tfrac{1}{2}]=\frac{2^n x\log^n|x|}{n!},
$$
if $\|x\|_2=1$. For arbitrary $x\in\ell_2$ we have, by homogeneity,
\begin{equation}\label{taylors}
\widehat F_x[n, \tfrac{1}{2}]=\frac{2^n x}{n!}\log^n\left(\frac{|x|}{\|x\|_2}\right).
\end{equation}
In particular,
$$
\begin{aligned}
\Omega_{1,n}(x)&=(F_x[n-1, \tfrac{1}{2}],\dots, F_x[1, \tfrac{1}{2}])\\
& = x\left(\frac{2^{n-1}}{(n-1)!}\log^{n-1}\left(\frac{|x|}{\|x\|_2}\right),\dots,  2\log^2\left(\frac{|x|}{\|x\|_2}\right), 2\log\left(\frac{|x|}{\|x\|_2}\right) \right)\\
\end{aligned}
$$
which allows us to describe the corresponding spaces $\mathscr Z^{(n)}$ for small $n$ as follows. First, we have
$$
\mathscr Z^{(2)}\approx \ell_2\oplus_{\Omega_{1,1}}\ell_2=\{(y,x): \|y-2x\log(|x|/\|x\|_2)\|_2+\|x\|_2 <\infty\},
$$
which is well isomorphic to Kalton-Peck $Z_2$ space \cite[Section~6]{kaltpeck}. Also,
$$
\mathscr Z^{(3)}\approx \mathscr Z^{(2)} \oplus_{\Omega_{1,2}} \ell_2 \approx \left( \ell_2\oplus_{\Omega_{1,1}}\ell_2  \right) \oplus_{\Omega_{1,2}} \ell_2,
$$
and the norm of $
\mathscr Z^{(3)}$ is equivalent to
\begin{equation}\label{z3}
\|(z,y,x)\|_{\Omega_{1,2}}= \left\| \left(z- 2x\log^2\left(\frac{|x|}{\|x\|_2}\right), y-2x\log\left(\frac{|x|}{\|x\|_2}\right) \right)\right\|_{\Omega_{1,1}}+ \|x\|_2.
\end{equation}

We will also finally display the quasilinear map $\Omega_{2,2}$
that allows one to represent $\mathscr Z^{(4)}$ as a twisted sum
of $\mathscr Z^{(2)}$ with itself. After all, this was the
starting point of this research. Let $\varphi:\mathbb S\to\mathbb
D$ be conformal equivalence vanishing at $z_0=1/2$ and let
$\phi=\sum_{1\leq i\leq 3}a_i\varphi^i$ be such that
$\widehat\phi[i;z={1\over 2}]=\delta_{1i}$ for $0\leq i\leq 3$.
Given $(y,x)\in \mathscr Z^{(2)}$ we construct an allowable
$F_{(y,x)}\in \mathscr F$ as follows. Let $F_x$ and
$F_{(y-\Omega(x))}$ be extremals for $x$ and $y-\Omega(x)$,
respectively, where $\Omega(x)=\Omega_{1,1}(x)=F_x'({1\over
2})=2x\log(|x|/\|x\|_2)$. Put
$$
G=\phi\cdot F_{(y-\Omega(x))}+ F_x.
$$
Then $G({1\over 2})=x, G'({1\over 2})= y$ and $\|G\|_\mathscr
F\leq  \|\phi\cdot F_{(y-\Omega(x))}\|_\mathscr F+
\|F_x\|_\mathscr F\leq \|\phi\|_\infty
(\|y-\Omega(x)\|_2+\|x\|_2)$, where $\|\phi\|_\infty\leq
|a_1|+|a_2|+|a_3|$ and we may define
$$
\Omega_{2,2}(y,x)=\left(\widehat F_{(y,x)}[3;\tfrac{1}{2}],
\widehat F_{(y,x)}[2;\tfrac{1}{2}]\right).
$$
By the construction of $\phi$ we have $\widehat
F_{(y,x)}[2;{1\over 2}]= \widehat F_{(y-\Omega(x))}[1;{1\over 2}]+
\widehat F_{x}[2;{1\over 2}]$ and $\widehat F_{(y,x)}[3;{1\over
2}]= \widehat F_{(y-\Omega(x))}[2;{1\over 2}]+ \widehat
F_{x}[3;{1\over 2}]$ and thus
$$%\begin{eqnarray*}
\Omega_{2,2}(y,x)=2\left((y-\Omega x)\log^2\frac{|y-\Omega x|}{\|y-\Omega x\|_2}+
\frac{2x}{3}\log^3\frac{|x|}{\|x\|_2}     , (y-\Omega
x)\log\frac{|y-\Omega x|}{\|y-\Omega x\|_2}+
{x}\log^2\frac{|x|}{\|x\|_2}      \right).
$$%\end{eqnarray*}
%We would be very proud of this formula, but it is due basically to Rochberg.

\subsection{The 3-space problem for twisted Hilbert spaces}
We are now ready for the first concrete application. Recall that a
twisted Hilbert space is a twisted sum of Hilbert spaces.

\begin{proposition}
The space $\mathscr Z^{(n)}$ is a twisted Hilbert space if and
only if $n=1,2$.
\end{proposition}

\begin{proof}
The $n$-th cotype 2 constant $a_{n,2}(X)$ of a (quasi-) Banach space $X$ is defined as the infimum of those $C$ such that
for every $x_1, \dots x_n\in X$ one has
$$\left ( \int_0^1 \left\|\sum_{i=1}^n r_i(t) x_i \right\|^2 dt \right)^{1/2} \leq
C \left(\sum_{i=1}^n \|x_i\|^2 \right)^{1/2},$$
where $(r_n)$ is the sequence of Rademacher functions.

To prove that $\mathscr Z^{(3)}$ does not embed in any
twisted Hilbert space we will work with the equivalent
quasinorm given by (\ref{z3}). Let $(e_i)$ be the unit basis of
$\ell_2$ and take $x_i = (0, 0,e_i)$. These are normalized
vectors, which makes $(\sum_{i=1}^n\|x_i\|^2)^{1/2} = \sqrt{n}$.
On the other hand
$$\left\|\sum_{i=1}^n \pm x_i \right\|_{\Omega_{1,2}} = \sqrt{n}(1+ \log^2n).$$ Hence the cotype 2
constants of $\mathscr Z^{(3)}$ cannot verify $a_{n,2}
\leq K \log n$. And this estimate must hold in any twisted Hilbert
space by \cite[Theorem 6.2.(a)]{kaltpeck}.
\end{proof}

\begin{corollary} ``To be a twisted Hilbert space'' is not a 3-space
property.\end{corollary}

The corollary answers a question posed to us by David Yost long
time ago \cite[p. 95]{castgonz} and considered by the first author
in \cite{cabetwisted} where it was shown that ``to be a subspace
of a twisted Hilbert space'' is not a 3-space property. Since
$\mathscr Z^{(n)}$ is isomorphic to its dual (see
\cite[Section~4]{rochberg}) and the dual of any twisted Hilbert
space is again a twisted Hilbert space, we see that $\mathscr
Z^{(n)}$ is a quotient of a twisted Hilbert space if and only if
$n=1,2$.

Thus, in the situation described in Section~\ref{KPR}, recall that
for $\mathscr F=\mathscr F(\ell_\infty,\ell_1)$ one gets $\mathscr
Z^{(1)}=\ell_2$ and   $\mathscr Z^{(2)}$ is isomorphic to
Kalton-Peck's space $Z_2$ and, actually, the extension
$0\to\ell_2\to \mathscr Z^{(2)}\to\ell_2\to 0$ is isomorphically
(and even ``projectively'' cf. \cite{kaltpeck}) equivalent to
Kalton-Peck's sequence $0\to \ell_2\to Z_2\to \ell_2\to0$, which
has strictly singular quotient map and strictly cosingular
inclusion (see \cite[Theorem 6.4]{kaltpeck}). One therefore has.

\begin{proposition} The exact sequences $0\to\mathscr
Z^{(k)}\to \mathscr Z^{(n+k)}\to\mathscr Z^{(n)}\to0$ are singular
and cosingular, for all integers $n,k$. \hfill$\square$
\end{proposition}

As a direct application we get:

\begin{proposition}\label{nocomp2} If $k>1$ the space $\mathscr Z^{(k)}$ does not contain complemented
copies of $\ell_2$.
\end{proposition}

\begin{proof}
By \cite[Corollary 6.7]{kaltpeck} $\mathscr Z^{(2)} = Z_2$ has no
complemented subspaces isomorphic to $\ell_2$. Now, if one has an
exact sequence
$$\begin{CD}
0@>>> A @>I>> B@>Q>> C @>>>0.\end{CD}
$$ with $Q$ strictly singular and $A$ not containing $\ell_2$ complemented then $B$ does not contain $\ell_2$ complemented:
assume otherwise that $B$ has a subspace $B'$ which is isomorphic to $\ell_2$ and is complemented in $B$ through a projection $P$. (Without loss of generality we may assume that $A=\ker Q$ and $I$ is the inclusion map.)
Since $Q$ is strictly singular, there exist an infinite dimensional subspace $A'\subset A$ an a nuclear operator $K:A'\to B$ such that $I-K: A'\to B'$ is an embedding. Passing to a further subspace if necessary we may assume the nuclear norm of $K$ is strictly less than 1.
Let $N$ be a  nuclear endomorphism of $B$ extending $K$ and having the same nuclear norm as $K$. Then $\|N:B\to B\|<1$ and ${\bf 1}_B-N$ is invertible, with $({\bf 1}_B-N)^{-1}=\sum_{k\geq 0}N^k$ -- summation in the operator norm. Now, it is easily seen that
$$
({\bf 1}_B-N)\circ P\circ  ({\bf 1}_B-N)^{-1}
$$
is a projection of $B$ (hence of $A$) onto $A'$.
\end{proof}

The proof also works replacing $\ell_2$ by any other
``complementably minimal'' space (those Banach spaces all whose
infinite dimensional closed subspaces contain subpaces isomorphic
to and complemented in the whole space) such as $\ell_p$ for
$1<p<\infty$. This implies that Proposition \ref{nocomp2} extends
almost verbatim for $1<p<\infty$.

\subsection{A twisted sum of $Z_2$ containing $\ell_2$ complemented}
It is quite surprising that there exists a twisted sum of $Z_2$
containing complemented copies of $\ell_2$. But they exist:

\begin{proposition} There is a (nontrivial) exact sequence
$$\begin{CD}
0@>>> Z_2 @>>> \ell_2 \oplus \Z^{(3)} @>>> Z_2 @>>>0.
\end{CD}
$$
\end{proposition}
\begin{proof} Recall from \cite[p.257]{castmoreisr} the construction of the so-called diagonal
push-out sequence: In a commutative diagram
$$
\begin{CD}
0@>>>A@>\imath >> B@>>> E@>>> 0\\
&& @V u VV @VV v V@|\\
0@>>>C@>\jmath>> D@>>> E@>>>0
\end{CD}
$$
the following sequence is exact
$$
\begin{CD}
0@>>> A@>{\imath\times u}>>B\oplus C@> v\ominus\jmath >>D@>>>0,
\end{CD}
$$
where $(\imath\times u)(a)=(\imath(a),u(a))$ and
$(v\ominus\jmath)(b,c)= v(b)-\jmath(c)$. Thus, taking $n=i=1$ and
$k=j=2$ in Diagram~\ref{poz} for $\mathscr F=\mathscr
F(\ell_\infty,\ell_1)$ one gets a commutative diagram
$$
\begin{CD}
0@>>> \mathscr Z^{(2)}@>>> \mathscr Z^{(3)}@>>> \ell_2 @>>> 0\\
&&@VVV @VVV@|\\
0@>>> \ell_2 @>>> \mathscr Z^{(2)}@>>>\ell_2 @>>> 0,
\end{CD}
$$
from which, recalling that $\Z^{(2)} =Z_2$, one obtains an exact sequence
$$
\begin{CD}
0@>>> Z_2@>>> \ell_2 \oplus \mathscr Z^{(3)}@>>> Z_2 @>>> 0
\end{CD}
$$
which is not trivial since otherwise $Z_2 \simeq Z_2\oplus Z_2=
\ell_2 \oplus \mathscr Z^{(3)}$, something impossible since $Z_2$
does not contain $\ell_2$ complemented.\end{proof}

 Therefore $\ell_2\oplus \mathscr Z^{(3)}$ is a
twisted sum of $Z_2$, which contains complemented Hilbert
subspaces. We cannot resist to remark that while nobody knows
whether $Z_2$ is isomorphic to its hyperplanes, it is obvious that
$\ell_2\oplus \mathscr Z^{(3)}$ is isomorphic to its own
hyperplanes.

\section{Open ends}

\subsection{On the splitting of the first extension}
Very little is known about the splitting of the ``first''
exact sequence $0\to\mathscr Z^{(1)}\to \mathscr Z^{(2)}\to\mathscr
Z^{(1)}\to0$ outside of the case in which it is induced by a
couple of Banach lattices. On the other hand, Corollary \ref{nonsplitting} shows that
once the first exact sequence obtained in an interpolation schema
is nontrivial, the same happens to all the rest. Is it true the
reciprocal? That is, suppose that $\mathscr Z^{(2)}$ is a trivial
self-extension of $\mathscr Z^{(1)}$. Does it follow that the
extensions $0\to\mathscr Z^{(k)}\to \mathscr Z^{(n+k)}\to\mathscr
Z^{(n)}\to 0$ are trivial for all values of $n$ and $k$?

\subsection{Other twisted Hilbert spaces}
Suppose we have a Banach space $X_0$
with a normalized basis that we use to consider $X_0$ inside
$\ell_\infty$. Take $X_1=\overline{X_0'}$ the complex conjugate of
the closure $X_0'$ of the subspace spanned by the coordinate
functionals in $X_0^*$. Then $(X_0,X_1)$ is a Banach couple, $[X_0,X_1]_{1/2}$ is a Hilbert space (see
\cite[around Theorem 3.1]{pisier-handbook}), and thus $\mathscr
Z^{(2)}$ is a twisted Hilbert space. We believe that $\mathscr
Z^{(2)}$ is a Hilbert space if and only if $X_0=\ell_2$.

\subsection{Other interpolation methods} Most of the work done here can be reproduced for real interpolation
 by either the $K$ or $J$ methods as it can be deduced from the results in this paper and those in \cite{ccs2}. It would be interesting to know to what extent the same occurs for other interpolation methods.

\subsection{About the vanishing of $\Ext^2$.}
A problem at the horizon, for us, was whether
the second derived functor $\Ext^2$ vanishes on Hilbert spaces, which can be
understood as a twisted reading of a question of Palamodov for Fr\'echet spaces (\cite[Section 12, Problem 6]{p}).

Given Banach spaces $A$ and $D$, one considers the set of all possible four-term exact sequences
\begin{equation}\label{four}
\begin{CD}
0@>>> A@> I>> B@>U>> C@>Q>> D@>>> 0.
\end{CD}
\end{equation}
Under a certain equivalence relation, which is not necessary to define here, the set of such four-term exact sequences becomes a linear space denoted by $\Ext^2(D,A)$, whose zero is (the class of all exact sequences equivalent to)
\begin{equation*}%\label{four}
\begin{CD}
0@>>> A@= A@>0>> D@= D@>>> 0.
\end{CD}
\end{equation*}

It is important to realize that if we are given a short exact sequence of the form
\begin{equation}\label{ABE}
\begin{CD}
0@>>> A@> I>> B@> P >> E @>>> 0
\end{CD}
\end{equation}
and another sequence of the form
\begin{equation}\label{ECD}
\begin{CD}
0@>>> E@> J >> C@> Q >> D @>>> 0
\end{CD}
\end{equation}
then we may form a four-term sequence
\begin{equation}\label{four}
\begin{CD}
0@>>> A@> I>> B@> U >> C@> Q >> D @>>> 0
\end{CD}
\end{equation}
just taking $U=J\circ P$. This resulting ``long'' sequence will be zero in $\Ext^2(D,A)$ if and only if (\ref{ABE}) and (\ref{ECD})
fit inside a commutative diagram
\begin{equation}\label{F2=0}
\begin{CD}
&&&&0&&0\\
&&& & @AAA @AAA\\
&&&&D @= D\\
&&& & @AAA @AAA\\
0@>>> A @>>> F @>>> C @>>> 0\\
& & @| @AAA @AAA\\
0@>>> A@>>> B@>>> E@>>> 0\\
&&& & @AAA @AAA\\
&&&&0&&0
\end{CD}
\end{equation}
whose rows and columns are exact.

The skeptical reader will wonder how is this related to the main subject of the paper.
Let $\mu$ a $\sigma$-finite measure on a measure space $S$
and let $L_0$ be the space of all (complex) measurable functions
on $S$, where we identify two functions if they agree almost
everywhere. If $X$ is a K\"othe space on $\mu$, then a centralizer
on $X$ is a homogeneous mapping $\Omega:X\to L_0$ having the
following property: there is a constant $C=C(\Omega)$  such that,
for every $f\in X$ and every $a\in L_\infty$, the difference
$\Omega(af)-a\Omega(f)$ belongs to $X$ and
$\|\Omega(af)-a\Omega(f)\|_X\leq C\|a\|_\infty\|f\|_X$. Every
centralizer is quasilinear and so it induces a twisted sum
$X\oplus_\Omega X=\{(y,x): x, y-\Omega(x)\in X\}$ which is quasinormed
by the functional $\|(y,x)\|_\Omega=\|y-\Omega(x)\|_X+\|x\|_X$. A
widely ignored result by Kalton states that if $X$ is
super-reflexive then one can construct an admissible space of
analytic functions $\mathscr F$ on a disc centered at the origin
such that:
\begin{itemize}
\item $X=X^{(1)}_0$ (evaluation at $0$) up to equivalent norm;
\item $\Omega\approx\Omega_{1,1}$, where $\Omega_{1,1}$ is the
corresponding ``derivation'' (see Section~\ref{twistedsum}).
\end{itemize}
This means that for every $x\in X$ the difference
$\Omega(x)-\Omega_{1,1}(x)$ falls in $X$ and one has the estimate
$\|\Omega(x)-\Omega_{1,1}(x)\|_X\leq K\|x\|_X$ for some constant $K$
and every $x\in X$. Actually one can construct $\mathscr F$ by
using no more than three K\"othe spaces on the boundary of the
disc \cite[Theorem 7.9]{kal-diff}; if $\Omega$ is ``real'' in the
sense that it takes real functions into real functions, then two
K\"othe spaces on a strip suffice \cite[Theorem 7.6]{kal-diff}. In
particular since $X\oplus_\Omega X= X\oplus_{\Omega_{1,1}} X=
X^{(2)}_0$, up to equivalent (quasi-) norms, we see that the
self-extension induced by $\Omega$ fits into the commutative diagram (the operators $\imath_{n,k}$ are those appearing in Proposition \ref{operators}):
$$
\begin{CD}
0@>>> X@>>> X^{(3)}_0@>>> X\oplus_\Omega X @>>> 0\\
& & @| @A {\imath_{2,3}}AA @A {\imath_{1,2}}AA\\
0@>>> X@>>> X\oplus_\Omega X@>>> X@>>> 0\\
\end{CD}
$$which, when completed, has the same form as (\ref{F2=0}) witnessing that the juxtaposition of two copies of the extension induced by $\Omega$, namely
$$
\begin{CD}
0@>>> X@>>>X\oplus_\Omega X@>>> X\oplus_\Omega X@>>> X@>>> 0,
\end{CD}
$$ is zero in $\Ext^2(X,X)$. We do not know what happens with two different centralizers; more specifically, we ask the following. Let $\Omega$ and $\Phi$ be centralizers on a super-reflexive K\"othe space $X$ and consider the twisted sums $X\oplus_\Omega X$ and $X\oplus_\Phi X$. If, as before, we set $I(x)=(x,0), U(x,y)=(y,0)$ and $Q(x,y)=y$, can the exact sequence
$$
\begin{CD}
0@>>> X@>I>>X\oplus_\Omega X@>U>> X\oplus_\Phi X@>Q>> X@>>> 0
\end{CD}
$$
be nonzero in $\Ext^2(X,X)$?

%We postpone until a forthcoming paper the full exposition of the theory around this problem and its many connections.

\end{document}